\documentclass[preprint,12pt]{elsarticle}

\usepackage{amssymb}
\usepackage{enumerate}
\usepackage{csquotes}
\usepackage{mathtools}
\usepackage{amssymb}
 \newtheorem{thm}{Theorem}[section]
 \newtheorem{theorem}{Theorem}[section]

 \newtheorem{corollary}[thm]{Corollary}
\newproof{proof}{\textbf{Proof}}
 
 \newtheorem{lemma}[thm]{Lemma}

 \newtheorem{ex}{Example}
 \numberwithin{equation}{section}
\usepackage{epsfig,subfigure}
\usepackage{epstopdf}
\usepackage{geometry}
\usepackage{graphicx}
\usepackage{graphics}
\usepackage{subfigure}
\usepackage{tikz}
\usepackage{hyperref}
\usepackage{multirow}
\allowdisplaybreaks

\begin{document}

\begin{frontmatter}



\title{On weighted Schatten $p$-numerical radii via generalized Aluthge transformation}



\author{Kallal Pal} 
\author{Sumit Chandok\corref{cor}} 
\cortext[cor]{Corrosponding author.}\ead{sumit.chandok@thapar.edu}
\affiliation{organization={Department of Mathematics,\\ Thapar Institute of Engineering and Technology},
            city={Patiala},
            postcode={147004}, 
            state={Punjab},
            country={India},
}
\begin{abstract}
In this article, we derive various weighted p-numerical radius inequalities using the Schatten p-norm with the generalized Aluthge transform for operators in Hilbert spaces. In addition, the article explores classical numerical radius and norm inequalities for operators. The bounds obtained in this work improve upon several well-known previous results.
In addition, we provide examples where we calculate and analyze upper bounds for the weighted numerical radius of different types of matrices using 
 generalized Aluthge transformation.
\end{abstract}





\end{frontmatter}
\section{Introduction and Preliminaries}
In modern analysis, mathematical inequalities are very important because they provide useful ways to predict and describe how operators act by using their upper and lower bounds. It's not only because of these inequalities that mathematical and functional analysis have grown, but also because they give us new ways to solve hard problems in science, engineering, and technology.
 In operator theory, the fact that operators are bounded is a key part of understanding their structural and functional properties. The numerical range and its related numerical radius are also very important for understanding their geometric and analytical properties. The relationships among the numerical radius, spectral radius, and operator norm enhance our understanding of how operators act on Hilbert spaces. These concepts have significant applications in fields such as approximation theory and numerical analysis. However, determining the exact radius of a general operator remains challenging. This complexity underscores the necessity for sharper bounds. In this framework, the Aluthge transform, introduced by Aluthge \cite{aluthge1990p}, has emerged as a powerful and versatile tool for obtaining these sharper bounds. Importantly, the Aluthge transform preserves key spectral properties while often improving the behavior of operators. This tool is very important in modern operator theory because it makes it easier to study topics such as convergence and norm inequalities. 
So, studying operator inequalities, numerical radius estimates, and Aluthge-type transforms together provides a rich and active framework for advancing the theoretical foundations and practical applications of functional analysis.
Let $H$ be a complex Hilbert space and let $B(H)$ be the $C^*-$algebra of all bounded linear operators on $H$.
\noindent For any compact operator $T \in B(H)$, the Schatten \(p\)-norm of \(T\) is defined as 
\[
\|T\|_p = \left( \operatorname{tr}|T|^p \right)^{\frac{1}{p}},
\]
where \(1 \leq p \leq \infty\) and
$|T| = (T^*T)^{\frac{1}{2}}$.
A compact operator $T$ belongs to the Schatten $p$-class, which is denoted by 
\(B_p(H)\), if $\|T\|_p < \infty$.\\
For $p=\infty$, the Schatten \(p\)-norm is the usual operator norm defined on \(B(H)\) as follows: \[ \|T\| = \sup_{\|x\|=1} \|Tx\|. \] For \(p = 2\), the Schatten \(p\)-norm is defined as the Hilbert--Schmidt norm: \[ \|T\|_2 = \left( \operatorname{tr} T^*T \right)^{\frac{1}{2}}. \]
For \(0 < p < 1\), \(\|\cdot\|_p\) constitutes a quasi-norm, as it fails to meet the triangle inequality. \\
For every \(T,S \in B_p(H)\), where \(0 < p \leq \infty\), we have the following relations (see \cite{frakis2025p} ):
\begin{align}\label{w0}
\|T\|_{rp}^{\,r}=\||T|^r\|_p=\||T^*|^r\|_p\qquad \text{for } r>0,   
\end{align}
and
\begin{align}\label{qe1}
\left\|
\begin{bmatrix}
T & 0 \\
0 & S
\end{bmatrix}
\right\|_p
=
\left\|
\begin{bmatrix}
0 & T \\
S & 0
\end{bmatrix}
\right\|_p
=
\begin{cases}
\left( \|T\|_p^p + \|S\|_p^p \right)^{\frac{1}{p}}
& \text{for } 0 < p < \infty, \\[1ex]
\max\{\|T\|,\|S\|\}
& \text{for } p = \infty.
\end{cases}
\end{align}
For \(1 \leq p \leq q \leq \infty\), the Schatten \(p\)-norm of \(T\) satisfies the monotonicity property
\[
\|T\|_\infty \leq \|T\|_q \leq \|T\|_p \leq \|T\|_1 .
\]
Moreover, if \(T \in B_p(H)\) and \(S \in B(H)\), then
\begin{align}\label{2}
\|TS\|_p \leq \|T\|_p \|S\|
\qquad \text{and} \qquad
\|ST\|_p \leq \|S\| \|T\|_p .
\end{align}
For $T \in B(H)$, the Aluthge transform \cite{aluthge1990p} of $T$, denoted as $\widetilde{T}$, is defined as
\[
\widetilde{T} = |T|^{\frac{1}{2}} U |T|^{\frac{1}{2}},
\]
where $|T| = (T^*T)^{\frac{1}{2}}$ and $U$ is the partial isometry associated with the polar decomposition of $T$ and so $T = U|T|$, $\ker T = \ker U$. 

Okubo \cite{okubo2003weakly} generalized the Aluthge transform, known as the $\mathsf{s}$-Aluthge transform, for $\mathsf{s} \in [0,1]$,
\[
\widetilde{T}_\mathsf{s} = |T|^\mathsf{s} U |T|^{1-\mathsf{s}}.
\]
Here, $|T|^0$ is defined as $U^*U$. In particular,
$\widetilde{T}_0 = U^*U^2|T|, \widetilde{T}_1 = |T|U, 
\widetilde{T}_{\frac{1}{2}} = |T|^{\frac{1}{2}} U |T|^{\frac{1}{2}} = \widetilde{T} \ (\text{the Aluthge transform of } T)$.

K. Shebrawi and M. Bakherad
\cite{shebrawi2018generalizations} introduced the generalized Aluthge transform of the operator $T$, denoted by $\widetilde{T}_{f,g}$ and defined by 
\[
\widetilde{T}_{f,g}=f(|T|)Ug(|T|),
\]
where $f$ and $g$ are non-negative continuous functions satisfying
$f(x)g(x)=x, x\geq0$.\\
Also, since $|T^*|^2=TT^*=U|T|^2U^*$  then 
$g(|T|)=U^*g(|T^*|)U$ for every 
non-negative continuous function $g$ on $[0,\infty)$.\\


Sheikhhosseini et al. \cite{sheikhhosseini2022weighted}, gave a generalized decomposition, referred to as the weighted real and imaginary parts of $T$, defined as:
\[
\Re_t(T) = (1 - t)T^* + tT \quad \text{and} \quad 
\Im_t(T) = \frac{(1 - t)T - tT^*}{i} \qquad \text{for } t \in [0,1].
\]
In \cite{conde2025some},
the weighted numerical radius is defined by
\[
w_t(T)  = \sup_{\theta \in \mathbb{R}} \|\Re_t(e^{i\theta}T)\|=\sup_{\theta \in \mathbb{R}}\|te^{i\theta}T + (1-t)e^{-i\theta}T^*\|.
\]
Clearly, for $t = \frac{1}{2}$, we have $\Re_t(T) = \Re(T)=\frac{T+T^*}{2}$ and $\Im_t(T) = \Im(T)=\frac{T-T^*}{2i}$, so the weighted numerical radius coincides with the classical numerical radius, that is,
\[w_\frac{1}{2}(T) = w(T) = \sup_{\theta \in \mathbb{R}} \|\Re(e^{i\theta}T)\|.\]

Zamani \cite{zamani2023weighted} defined the extension of the classical numerical radius as, for \( T \in B_p(H) \), we have:
\[
w_{t,p}(T)=\sup_{\theta\in\mathbb{R}}
\left\|
\Re_t\!\left(e^{i\theta}T\right)
\right\|_p, ~\text{where}~ 1\leq p\leq \infty.
\]
In particular, since the Schatten p-norm coincides with the usual operator norm for $p=\infty$, we get
 $  w_{t,p}(T)=w_t(T)$ 
Thus, $ w_{t,p}(\cdot)$ extends the concept of the weighted numerical radius given by Sheikhhosseini et al.  \cite{sheikhhosseini2022weighted}.

\noindent Also,
$w_{0,p}(T)=\|T^*\|_p, w_{1,p}(T)=\|T\|_p ~\text{and}~ w_{\frac{1}{2},p}(T)=w_p(T).$ 

\noindent For \(1 \leq p \leq \infty\), the \(p\)-numerical radius of \(T \in B_p(H)\) is defined by
\[
w_{\frac{1}{2},p}(T)=w_p(T)=\sup_{\theta\in\mathbb{R}}
\left\|
\Re\!\left(e^{i\theta}T\right)
\right\|_p .
\]

\noindent For the case \(0<p<1\), \(w_p(\cdot)\) is defined in a similar manner. 
It is well known that for \(1\leq p\leq \infty\), we have
\[
\frac{1}{2}\|T\|_p \leq w_p(T)\leq \|T\|_p .
\]
For \(p=2\) and \(p=\infty\), $ w_p(T)$ is the Hilbert--Schmidt numerical radius
$w_2(T)=
\sqrt{\frac{1}{2}\|T\|_2^2+
\frac{1}{2}\left|\operatorname{tr} T^2\right|}$,
(see \cite{abu2019generalization}), and the usual numerical radius \(w(T)\), respectively.

It is important to note that \(w_p(\cdot)\) is weakly unitarily invariant. This means that for every \(T\in B_p(H)\) and for every unitary operator \(U\in B(H)\), we get $w_p(UTU^*)=w_p(T).$ 
In addition, if \(T\) is self-adjoint, then
$w_p(T)=\|T\|_p $.

 It follows easily from the definition of $\widetilde{T}$ that $\|\widetilde{T}\| \leq \|T\|$ and  $w(\widetilde{T}) \leq w(T)$ (see \cite{jung2000aluthge}).
 
Using the Aluthge transform, Yamazaki \cite{yamazaki2007upper} proved that if $T \in B(H)$, then
\begin{align}\label{Q1}
w(T) \leq \frac{1}{2} \left( \|T\| + w(\widetilde{T}) \right). 
\end{align}
Abu-Omar and Kittaneh \cite{abuomar2013numerical} improved the inequality \eqref{Q1} using $\mathrm{s}$-Aluthge transform 
\begin{align*}
w(T) \leq \frac{1}{2} \left( \|T\| + \min_{t \in [0,1]} w(\widetilde{T}_\mathsf{s}) \right).
\end{align*}

Motivated by these results, we provide upper bounds for the weighted \(p\)-numerical radius of a bounded linear operator \(T\) within a complex Hilbert space, employing the Schatten \(p\)-norm. We use the polar decomposition \(T = U|T|\) together with the generalized Aluthge transform. These bounds generalized certain numerical radius inequalities that are already established in the literature. Specifically, when \(p = \infty\), the Schatten \(p\)-norm reduces to the standard operator norm, allowing us to present some classical weighted numerical radius inequalities for operators. Additionally, we provide examples where we calculate and analyze upper bounds for the weighted \(p\)-numerical radius of specific types of matrices through the generalized Aluthge transform.\\

\vspace{.1cm}
\noindent In order to give our results, we need the following lemmas. 
\begin{lemma}\label{lem1}  \cite{kittaneh1992arithmetic}
  Let $T, S \in B_p(H)$ be such that TS is self-adjoint, and let $1\leq p\leq \infty$. Then
$\|TS\|_p \leq \|\Re(ST)\|_p$.   
\end{lemma}
\begin{lemma}\label{l3}\cite{kittaneh1993norm}
Let \(A,B,X \in B(H)\) be such that \(A\) and \(B\) are positive and
\(X \in B_p(H)\), where \(1 \le p \le \infty\), and let
\(0 \le t \le 1\). Then
\[
\|A^t X B^{1-t}\|_p
\le
\|AX\|_p^t \|XB\|_p^{1-t}.
\]
\end{lemma}

\begin{lemma}\cite{kittaneh1997operator}\label{lt1}
Let $T,S\in B_p(H)$ be positive, and let $p\ge1$. Then
\[
\|T+S\|_p
\le
\left(\|T\|_p^p+\|S\|_p^p\right)^{\frac1p}
+2^{\frac1p}\bigl\|T^{\frac12}S^{\frac12}\bigr\|_p .
\]
\end{lemma}

\begin{lemma}\label{cv1}\cite{Be2021on}
Let $T\in \mathsf{B}_p(\mathsf{H})$ be such that $T^2 = 0$. Then
\[w(T)=2^{\frac{1}{p}-1}\|T\|_p.\]
\end{lemma}

\section{Main Results}
 We discuss the upper bounds of the weighted $p$-numerical radius by using the
$\mathsf{s}$-Aluthge transform.

\begin{theorem}\label{1e1}
 For $1 \leq p\leq \infty$, $t \in [0,1]$ and $T \in B_p(H)$,  if $f,g$ are non-negative continuous functions on $[0,\infty)$ such that $f(x)g(x)=x$, then 
we have  
\begin{align}\label{0c1}
w_{t,p}(T)\leq2^{\frac1p}t\left[
w_p(\widetilde{T}_{f,g})
+
\frac{1}{2}
\left\|
f^2(|T|)+g^2(|T|)
\right\|_p\right]+|1-2t|\|T\|_p.
\end{align}
\end{theorem}
\begin{proof}
Let $T=U|T|$ be the polar decomposition of $T$. 
Consider,
\begin{align}
\|\Re_t(e^{i\theta}T)\|_p
\nonumber &=\left\|
te^{i\theta}T + (1-t)e^{-i\theta}T^*
\right\|_p\\
\nonumber &=\left\|
te^{i\theta}T +te^{-i\theta}T^*+ (1-2t)e^{-i\theta}T^*
\right\|_p\\
\nonumber &=\left\|
te^{i\theta}T +te^{-i\theta}T^*\|_p+ \|(1-2t)e^{-i\theta}T^*
\right\|_p\\
\nonumber &=
t\left\|e^{i\theta}U|T|+e^{-i\theta}|T|U^*\right\|_p + |1-2t|\|e^{-i\theta}|T|U^*\|\\
\nonumber &=
t\left\|
\begin{bmatrix}
e^{i\theta}Ug(|T|) & f(|T|)\\
0 & 0
\end{bmatrix}
\begin{bmatrix}
f(|T|) & 0\\
e^{-i\theta}g(|T|)U^* & 0
\end{bmatrix}
\right\|_p+|1-2t|\||T|\|_p
\end{align}
Here it is easy to see that the product $\begin{bmatrix}
e^{i\theta}Ug(|T|) & f(|T|)\\
0 & 0
\end{bmatrix}
\begin{bmatrix}
f(|T|) & 0\\
e^{-i\theta}g(|T|)U^* & 0
\end{bmatrix}$ is self adjoint. Since $T \in B_p(H)$, applying the Lemma \ref{lem1}, we have
\begin{align}\label{k2}
\nonumber \|\Re_t(e^{i\theta}T)\|_p&\le t
\left\|
\Re\left(
\begin{bmatrix}
f(|T|) & 0\\
e^{-i\theta}g(|T|)U^* & 0
\end{bmatrix}
\begin{bmatrix}
e^{i\theta}Ug(|T|) & f(|T|)\\
0 & 0
\end{bmatrix}
\right)
\right\|_p+|1-2t|\|T\|_p\\
\nonumber &=
t\left\|\Re\left(
\begin{bmatrix}
e^{i\theta}f(|T|)Ug(|T|)
&
f^2(|T|)
\\[2mm]
g^2(|T|)
&
e^{-i\theta}g(|T|)U^*f(|T|)
\end{bmatrix}\right)
\right\|_p+|1-2t|\|T\|_p\\
\nonumber &=t
\left\|
\begin{bmatrix}
\Re(e^{i\theta}\widetilde{T}_{f,g})
&
\dfrac12\bigl(f^2(|T|)+g^2(|T|)\bigr)
\\[2mm]
\dfrac12\bigl(f^2(|T|)+g^2(|T|)\bigr)
&
\Re\left(e^{-i\theta}(\widetilde{T}_{f,g})^*\right)
\end{bmatrix}
\right\|_p+|1-2t|\|T\|_p\\
\nonumber &\le t
\left\|
\begin{bmatrix}
\Re(e^{i\theta}\widetilde{T}_{f,g}) & 0\\
0 & \Re\left(e^{-i\theta}(\widetilde{T}_{f,g})^*\right)
\end{bmatrix}
\right\|_p\\
\nonumber &\hspace{1cm}+\frac{1}{2}t
\left\|
\begin{bmatrix}
0 &
\bigl(f^2(|T|)+g^2(|T|)\bigr)
\\[2mm]
\bigl(f^2(|T|)+g^2(|T|)\bigr)
& 0
\end{bmatrix}
\right\|_p+|1-2t|\|T\|_p\\
\nonumber &=
2^{\frac1p}t
\left\|
\Re(e^{i\theta}\widetilde{T}_{f,g})
\right\|_p
+
2^{\frac1p-1}t
\left\|
f^2(|T|)+g^2(|T|)
\right\|_p+|1-2t|\|T\|_p\\
\nonumber &\hspace{5cm}(\text{Using} ~\eqref{qe1})\\
 &=
2^{\frac1p}t
\left[\left\|
\Re(e^{i\theta}\widetilde{T}_{f,g})
\right\|_p
+
\frac{1}{2}
\left\|
f^2(|T|)+g^2(|T|)
\right\|_p\right]+|1-2t|\|T\|_p.
\end{align}
By taking the supremum on both sides over $\theta \in \mathbb{R}$, we get 
\begin{align*}
w_{t,p}(T)\leq2^{\frac1p}t\left[
w_p(\widetilde{T}_{f,g})
+
\frac{1}{2}
\left\|
f^2(|T|)+g^2(|T|)
\right\|_p\right]+|1-2t|\|T\|_p.
\end{align*}
\end{proof}
Taking $t=\frac{1}{2}$ in \eqref{0c1}, we get the following result of S. Sahoo \cite{sahoo2022further}.
\begin{corollary}
 For $1 \leq p\leq \infty$, $t \in [0,1]$ and $T \in B_p(H)$,  if $f,g$ are non-negative continuous functions on $[0,\infty)$ such that $f(x)g(x)=x$, then 
we have  
\begin{align*}
w_{t,p}(T)\leq2^{\frac1p-1}\left[
w_p(\widetilde{T}_{f,g})
+
\frac{1}{2}
\left\|
f^2(|T|)+g^2(|T|)
\right\|_p\right].
\end{align*}    
\end{corollary}
\noindent Taking $g(x)=x^\mathsf{s}$ and $f(x)=x^{1-\mathsf{s}}$ in \eqref{0c1}, we get the following result.
\begin{theorem}\label{e1}
 For $1 \leq p\leq \infty$, $\mathsf{s},t \in [0,1]$ and $T \in B_p(H)$, we have  
\begin{align}\label{c0}
w_{t,p}(T)\leq2^{\frac1p}t\left[
w_p(\widetilde{T}_{1-\mathsf{s}})
+
\frac{1}{2}
\left\|
|T|^{2\mathsf{s}}+|T|^{2(1-\mathsf{s})}
\right\|_p\right]+|1-2t|\|T\|_p.
\end{align}
\end{theorem}
\begin{corollary}\label{co1}
Let $T\in B_p(H)$ and $1\le p\le\infty$. For $\mathsf{s}=\frac{1}{2}$ in \eqref{c0}, we get
\begin{align}
 w_{t,p}(T)\leq 2^\frac{1}{p}t\left[
w_p(\widetilde{T})
+
\left\|T
\right\|_p\right]+|1-2t|\|T\|_p.
\end{align}    
\end{corollary}
\noindent Taking $p=\infty$ in the Theorem \ref{e1}, we get
\begin{corollary}
For $T \in B(H)$ and $\mathsf{s}, t \in [0,1]$, we get  
\begin{align}\label{c010}
w_t(T)\leq t
\left[w(\widetilde{T}_{1-\mathsf{s}})
+
\frac{1}{2}
\left\|
|T|^{2\mathsf{s}}+|T|^{2(1-\mathsf{s})}
\right\|\right]+|1-2t|\|T\|.
\end{align}
\end{corollary}
\noindent Taking $t=\frac{1}{2}$ and $\mathsf{s}=\frac{1}{2}$ in \eqref{c010}, we have get the following result of Abu-Omer and Kittaneh   \cite{abuomar2013numerical}.
 \begin{corollary}
 For $T \in B(H)$,  
$w(T)\leq\frac{1}{2}
(w(\widetilde{T})
+
\left\|T\right\|)$.    
 \end{corollary}

\noindent The inequality \eqref{c0} holds for all $\mathsf{s}\in [0,1]$, and so by taking the minimum over $\mathsf{s}$, we obtain the following result.
\begin{theorem}
  For $1 \leq p\leq \infty$, $\mathsf{s},t \in [0,1]$ and $T \in B_p(H)$, we have
\begin{align}\label{c1}
w_{t,p}(T)\leq2^{\frac1p}t
\left[\min_{0\le \mathsf{s}\le 1}
w_p(\widetilde{T}_{1-\mathsf{s}})
+\frac{1}{2}
\left\|
|T|^{2\mathsf{s}}+|T|^{2(1-\mathsf{s})}
\right\|_p\right]+|1-2t|\|T\|_p.
\end{align}
\end{theorem}
\noindent Taking $t=\frac{1}{2}$ in \eqref{c1}, we get the following result of Frakis et al. \cite{frakis2025p}.
\begin{corollary}
 For $1 \leq p\leq \infty$, $\mathsf{s} \in [0,1]$ and $T \in B_p(H)$, we get
\begin{align}
w_{p}(T)\leq2^{\frac{1}{p}-1}
\min_{0\le \mathsf{s}\le 1}
w_p(\widetilde{T}_{1-\mathsf{s}})
+
2^{\frac1p-2}
\left\||T|^{2\mathsf{s}}+
|T|^{2(1-\mathsf{s})}
\right\|_p.    
\end{align}
\end{corollary}

\noindent Taking $p=\infty$ in \eqref{c1}, we get the following result.
\begin{corollary} 
For $T\in B(H)$ and $\mathsf{s},t \in [0,1]$, we obtain
\begin{align}\label{c2}
w_t(T)
\le
t\left[\min_{0\le \mathsf{s}\le 1}
w(\widetilde{T}_{1-\mathsf{s}})
+
\frac{1}{2}
\left\|
|T|^{2\mathsf{s}}+|T|^{2(1-\mathsf{s})}
\right\|\right]+|1-2t|\|T\|.
\end{align}    
\end{corollary}
For $t=\frac{1}{2}$ in \eqref{c2}, we have the following result of Bag et al. \cite{bag2020bounds}.
\begin{corollary}
For $T\in B(H)$ and $\mathsf{s} \in [0,1]$, we get
\begin{align}\label{c3}
w(T)
\le
\frac{1}{2}
w(\widetilde{T}_{1-\mathsf{s}})
+
\frac{1}{4}
\left\|
|T|^{2(1-\mathsf{s})}+|T|^{2\mathsf{s}}
\right\|.    
\end{align}
\end{corollary}
For $t=\frac{1}{2}$ and $\mathsf{s}=\frac{1}{2}$ in \eqref{c2}, we get the following result of Yamazaki  \cite{yamazaki2007upper}.
\begin{corollary}
For $T\in B(H)$,
$w(T)
\le
\frac{1}{2}
\Big(\min_{0\le \mathsf{s}\le 1}w(\widetilde{T})+\left\|
T\right\|\Big)$.  
\end{corollary}

\begin{theorem}\label{0t11}
 For $1 \leq p\leq \infty$, $t \in [0,1]$ and $T \in B_p(H)$,  if $f,g$ are non-negative continuous functions on $[0,\infty)$ such that $f(x)g(x)=x$, 
we have  
\begin{align*}
w_{t,2p}^2(T)
&\le
t^2\left\|
g(|T|)\widetilde{T}_{f,g}f(|T|)
\right\|_{p}
+
(1-t)^2\left\|
f(|T|)(\widetilde{T}_{f,g})^*g(|T|)
\right\|_{p}\\
&\hspace{2cm}+
t(1-t)\|T^*T+TT^*\|_{p}.
\end{align*}
\end{theorem}
\begin{proof}
Let \(T=U|T|\) be the polar decomposition of \(T\), and let
\(\theta\in\mathbb{R}\).
Consider,
\begin{align}
\nonumber (\Re_t(e^{i\theta}T))^2&=(te^{i\theta}T + (1-t)e^{-i\theta}T^*)^2 \\
&=t^2 e^{2i\theta}T^2 + (1-t)^2e^{-2i\theta}T^{*2} + t(1-t)(T^*T + TT^*)\label{1w}\\
\nonumber &= t^2e^{2i\theta}U|T|U|T| + (1-t)^2e^{-2i\theta}|T|U^*|T|U^* +t(1-t)(T^*T + TT^*)\\
\nonumber &= t^2e^{2i\theta}Ug(|T|)f(|T|) Ug(|T|)f(|T|) + (1-t)^2e^{-2i\theta}f(|T|)g(|T|)U^*f(|T|) g(|T|)U^* \\
\nonumber &\hspace{4cm}
+ t(1-t)(T^*T + TT^*)\\
\nonumber &= t^2e^{2i\theta}Ug(|T|)\Tilde{T}_{f,g}~f(|T|) + (1-t)^2e^{-2i\theta}f(|T|)(\Tilde{T}_{f,g})^*g(|T|)U^* 
\\ &\hspace{5cm}+ t(1-t)(T^*T + TT^*)\label{1w0}.
\end{align}
Thus,
\begin{align*}
\left\|
\bigl(\Re_t(e^{i\theta}T)\bigr)^2\right\|_{p}
&\le
t^2\left\|
Ug(|T|)\widetilde{T}_{f,g}f(|T|)
\right\|_{p}
+
(1-t)^2\left\|
f(|T|)(\widetilde{T}_{f,g})^*g(|T|)U^*
\right\|_{p}\\
&\hspace{2cm}+
t(1-t)\|T^*T+TT^*\|_{p}\\
&\le
t^2\left\|
g(|T|)\widetilde{T}_{f,g}f(|T|)
\right\|_{p}
+
(1-t)^2\left\|
f(|T|)(\widetilde{T}_{f,g})^*g(|T|)\right\|_{p}\\
&\hspace{2cm}+
t(1-t)\|T^*T+TT^*\|_{p}.
\end{align*}
By taking the supremum in the above inequality over
\(\theta\in\mathbb{R}\) and use \eqref{w0}, we get 
\begin{align*}
w_{t,2p}^2(T)
&\le
t^2\left\|
g(|T|)\widetilde{T}_{f,g}f(|T|)
\right\|_{p}
+
(1-t)^2\left\|
f(|T|)(\widetilde{T}_{f,g})^*g(|T|)
\right\|_{p}\\
&\hspace{2cm}+
t(1-t)\|T^*T+TT^*\|_{p}.
\end{align*}
\end{proof}
Taking $t=\frac{1}{2}$ in The Theorem \ref{0t11}, we get the following result of S. Sahoo \cite{sahoo2022further}.
\begin{corollary}
 For $1 \leq p\leq \infty$, $t \in [0,1]$ and $T \in B_p(H)$,  if $f,g$ are non-negative continuous functions on $[0,\infty)$ such that $f(x)g(x)=x$, 
we have  
 \begin{align*}
w_{t,2p}^2(T)
&\le
\frac{1}{4}\Big(\left\|
g(|T|)\widetilde{T}_{f,g}f(|T|)
\right\|_{p}
+
\left\|
f(|T|)(\widetilde{T}_{f,g})^*g(|T|)
\right\|_{p}+
\|T^*T+TT^*\|_{p}\Big).
\end{align*}   
\end{corollary}

\begin{theorem}\label{t11}
 For $1 \leq p\leq \infty$, $\mathsf{s},t \in [0,1]$ and $T \in B_p(H)$, we have
\begin{align}\label{t5}
w_{t,2p}^2(T)
\le
(1-2t+2t^2) \|T\|_{p}
\|\widetilde{T}_\mathsf{s}\|
+
t(1-t)\|T^*T+TT^*\|_{p}.
\end{align}
\end{theorem}
\begin{proof}
Let \(T=U|T|\) be the polar decomposition of \(T\), and let
\(\theta\in\mathbb{R}\).
Consider, $g(x)=x^{1-\mathsf{s}}$ and $f(x)=x^\mathsf{s}$ in \eqref{1w0}, we have
\begin{align}
\nonumber (\Re_t(e^{i\theta}T))^2
&\leq
 t^2e^{2i\theta}U|T|^{1-\mathsf{s}}\Tilde{T}_\mathsf{s}|T|^\mathsf{s} + (1-t)^2e^{-2i\theta}|T|^\mathsf{s}(\Tilde{T}_\mathsf{s})^*|T|^{1-\mathsf{s}}U^* 
\\\nonumber &\hspace{5cm}+ t(1-t)(T^*T + TT^*).
\end{align}
Thus,
\begin{align*}
\left\|
\bigl(\Re_t(e^{i\theta}T)\bigr)^2\right\|_{p}
&\le
t^2\left\|
U|T|^{1-\mathsf{s}}\widetilde{T}_\mathsf{s}|T|^\mathsf{s}
\right\|_{p}
+
(1-t)^2\left\|
|T|^\mathsf{s}(\widetilde{T}_\mathsf{s})^*|T|^{1-\mathsf{s}}U^*
\right\|_{p}\\
&\hspace{2cm}+
t(1-t)\|T^*T+TT^*\|_{p}\\
&\le
t^2\left\|
|T|^{1-\mathsf{s}}\widetilde{T}_\mathsf{s}|T|^\mathsf{s}
\right\|_{p}
+
(1-t)^2\left\|
|T|^\mathsf{s}(\widetilde{T}_\mathsf{s})^*|T|^{1-\mathsf{s}}
\right\|_{p}\\
&\hspace{2cm}+
t(1-t)\|T^*T+TT^*\|_{p}
\qquad
\text{(by \eqref{2})}\\
&\le
t^2\Big\||T|\widetilde{T}_\mathsf{s}\Big\|_{p}^{\,1-\mathsf{s}}
\,
\Big\|\widetilde{T}_\mathsf{s}|T|\Big\|_{p}^{\,\mathsf{s}}
+
(1-t)^2\Big\||T|(\widetilde{T}_\mathsf{s})^*\Big\|_{p}^{\,\mathsf{s}}
\,
\Big\|(\widetilde{T}_\mathsf{s})^*|T|\Big\|_{p}^{\,1-\mathsf{s}}\\
&\hspace{2cm}+
t(1-t)
\|T^*T+TT^*\|_{p}
\qquad\qquad
\text{(by Lemma \ref{l3})}\\
&\le
\{t^2+(1-t)^2\}
\|T\|_{p}
\|\widetilde{T}_\mathsf{s}\|+
t(1-t)
\|T^*T+TT^*\|_{p}
\qquad
\text{(by \eqref{2}}\\
&=
(1-2t+2t^2)
\|T\|_{p}
\|\widetilde{T}_\mathsf{s}\|+
t(1-t)
\|T^*T+TT^*\|_{p}.
\end{align*}
By taking the supremum in the above inequality over
\(\theta\in\mathbb{R}\) and use \eqref{w0}, we get 
\begin{align*}
w_{t,2p}^2(T)
\le
(1-2t+2t^2) \|T\|_{p}
\|\widetilde{T}_\mathsf{s}\|
+
t(1-t)\|T^*T+TT^*\|_{p}.
\end{align*}
\end{proof}
\noindent Taking $\mathsf{s}=\frac{1}{2}$ in \eqref{t5}, we obtain the following results.
\begin{corollary}
 For $1 \leq p\leq \infty$, $t \in [0,1]$ and $T \in B_p(H)$, we get
    \begin{align}\label{qe3}
w_{t,2p}^2(T)
\le
(1-2t+2t^2) \|T\|_{p}
\|\widetilde{T}\|
+
t(1-t)\|T^*T+TT^*\|_{p}.
\end{align}
If we take \(\widetilde{T}=0\) in \eqref{qe3}, then we get the following results.
\begin{corollary}
For $T\in B_p(T)$,
$ w_{t,2p}^2(T)
\le
t(1-t)
\|T^*T+TT^*\|_{p}$.   
\end{corollary}
\noindent For $t=\frac{1}{2}$ in \eqref{qe3}, we have
\begin{align}\label{t09}
w_{2p}^2(T)
\le
\frac{1}{2} \|T\|_{p}
\|\widetilde{T}\|
+
\frac{1}{4}\|T^*T+TT^*\|_{p}.
\end{align}
\end{corollary}

\noindent The Theorem \ref{t11} holds for all $\mathsf{s}\in [0,1]$, and so taking minimum we get
\begin{theorem}
 For $1 \leq p\leq \infty$, $\mathsf{s},t \in [0,1]$ and $T \in B_p(H)$, we have
\begin{align}\label{c022}
w_{t,2p}^2(T)
\le
(1-2t+2t^2) \|T\|_{p}
\min_{0\le \mathsf{s}\le1}\|\widetilde{T}_\mathsf{s}\|
+
t(1-t)\|T^*T+TT^*\|_{p}.
\end{align}   
\end{theorem}
\begin{corollary}
Let $T\in B_p(H)$, $1\le p\le\infty$, and let $0\le t\le1$. Then
\[
w_{t,2p}^2(T)
\leq (1-2t+2t^2)
\|T\|_{p}
\min_{0\le \mathsf{s}\le1}\big\|\widetilde{T}_\mathsf{s}\big\|+2^\frac{1}{p}t(1-t)\big(\|T\|^2_p+\|T^2\|_{p}\big).
\]
\end{corollary}
\begin{proof}
From Theorem \ref{t11}, we have
\begin{align*}
w_{t,2p}^2(T)&\leq (1-2t+2t^2)
\|T\|_{p}
\|\widetilde{T}_\mathsf{s}\|+
t(1-t)
\||T|^2+|T^*|^2\|_{p}\\
&\leq (1-2t+2t^2)
\|T\|_{p}
\|\widetilde{T}_\mathsf{s}\|\\
&\hspace{2cm}+
t(1-t)
\left(
\left(\||T|^2\|_{p}^p
+\||T^{*}|^2\|_{p}^p\right)^{1/p}
+2^{1/p}\big\||T||T^*|\big\|_{p}
\right)\\
&\hspace{6cm} (\text{using the Lemma \ref{lt1}})\\
&\leq (1-2t+2t^2)
\|T\|_{p}
\min_{0\le \mathsf{s}\le1}\big\|\widetilde{T}_\mathsf{s}\big\|+2^\frac{1}{p}t(1-t)\big(\|T\|^2_p+\|T^2\|_{p}\big).
\end{align*}
\end{proof}
\noindent Taking $t=\frac{1}{2}$ in \eqref{c022}, we get the following result of Frakis et al. \cite{frakis2025p}
\begin{corollary}
 For $1 \leq p\leq \infty$, $\mathsf{s} \in [0,1]$ and $T \in B_p(H)$, we get
\begin{align}\label{c023}
w_{2p}^2(T)
\le
\frac{1}{2} \|T\|_{p}
\min_{0\le \mathsf{s}\le1}\|\widetilde{T}_\mathsf{s}\|
+
\frac{1}{4}\|T^*T+TT^*\|_{p}.
\end{align} 
\end{corollary}
\noindent Taking $p=\infty$ and $t=\frac{1}{2}$ in \eqref{c022}, we have the following result of Bag et al. \cite{bag2020bounds}.
\begin{corollary}
For $T \in B(H)$ and $\mathsf{s} \in [0,1]$, we get
\begin{align*}
 w^2(T)
\le
\frac{1}{2} \|T\|
\min_{0\le \mathsf{s}\le1}\|\widetilde{T}_\mathsf{s}\|
+
\frac{1}{4}\|T^*T+TT^*\|.   
\end{align*}    
\end{corollary}

\noindent Taking $p=\infty$ in the Theorem \ref{t11}, we get the following result.
\begin{corollary}
For $T \in B(H)$ and $\mathsf{s}, t\in [0,1]$, we have 
\begin{align}\label{f1}
w_t^2(T) \leq (1-2t+2t^2)\|T\|\|\widetilde{T}_\mathsf{s}\| + t(1-t)\|T^*T + TT^*\|.
\end{align}
\end{corollary}
We know that $\|\widetilde{T}\| \leq \|T^2\|^\frac{1}{2}$ (see [\cite{yamazaki2007upper}, (2.1)]) and $\|T^*T +TT^*\| \leq \|T\|^2+\|T^2\|$ (see
[\cite{kittaneh2002commutator}, Lemma 7]). 
Taking $t = \frac{1}{2}$ and $\mathsf{s}=\frac{1}{2}$ in \eqref{f1}, we get the following result of Kittaneh \cite{kittaneh2003numerical}.
\begin{corollary}
For $T \in B(H)$,
$w_t^2(T) \leq \frac{1}{2}(\|T\|+ \|T^2\|^\frac{1}{2})$.\\

\end{corollary}

In the following example, we discuss the bounds for the invertible, nilpotent, and diagonalizable matrices
\begin{ex}
Let $T\in B(H)$. For $p=\infty$, the Schatten $p$-norm coincides with the usual operator (spectral) norm. That is,
\[
\|T\|_{\infty}=\|T\|
=\sup_{\|x\|=1}\|Tx\|
=\sigma_{\max}(T)
=\sqrt{\lambda_{\max}(T^{*}T)},
\]
where $\sigma_{\max}(T)$ denotes the largest singular value of $T$, and $\lambda_{\max}(T^{*}T)$ is the largest eigenvalue of the positive operator $T^{*}T$. \\
Suppose that,
\begin{table}[h]
\centering
\label{tab:1}

\renewcommand{\arraystretch}{2.5}
\setlength{\tabcolsep}{10pt}

\begin{tabular}{|c||c|c||c|}
\hline
\multirow{2}{*}{Matrix}
& \multicolumn{2}{c||}{Diagonalizable Matrices}
& \multirow{2}{*}{Nilpotent Matrix} \\ \cline{2-3}

& Invertible & Non-invertible & \\

\hline

$T$
&
$T_1$=
$\begin{pmatrix}
0 & 1 & 0\\
0 & 0 & 2\\
1 & 0 & 0
\end{pmatrix}$
&
$T_2=
\begin{pmatrix}
0 & 2 & 0\\
0 & 0 & 1\\
0 & 0 & 0
\end{pmatrix}$
&
$T_3=\begin{pmatrix}
0 & 3 & 0\\
0 & 0 & 0\\
0 & 2 & 2
\end{pmatrix}$

\\
\hline

\end{tabular}
\end{table}
\newpage
Then, the upper bounds for the weighted numerical radius $w_t(T)$ and $w_t^2(T)$ in \eqref{c010} and \eqref{f1} are as shown in the Figure \ref{fig:01} and Figure \ref{fig:02} respectively.

Equations \eqref{c010} and \eqref{f1} provide parametric upper bounds for the numerical radii \( w(T) \) and \( w^2(T) \) respectively of a bounded linear operator \( T \) on a Hilbert space. These bounds are expressed in terms of the \( s \)-Aluthge transform with polar decomposition \( T = U|T| \). By varying the parameters \( s \in [0, 1] \) and \( t \in [0, 1] \), we visually represent how the bounds change together with these parameters. The resulting surface offers a clear visualization of the variation of the proposed upper bounds over the entire parameter region \( [0, 1]^2 \). 
Specifically, the plot allows us to identify the parameter ranges where the bounds are sharper, demonstrating the flexibility of the inequality concerning the choices of \( s \) and \( t \). This graphical representation complements the theoretical estimates by illustrating their numerical behavior for different combinations of parameters.

\begin{figure}[h]
    \centering
\includegraphics[width=4.5cm, height=4.5cm]{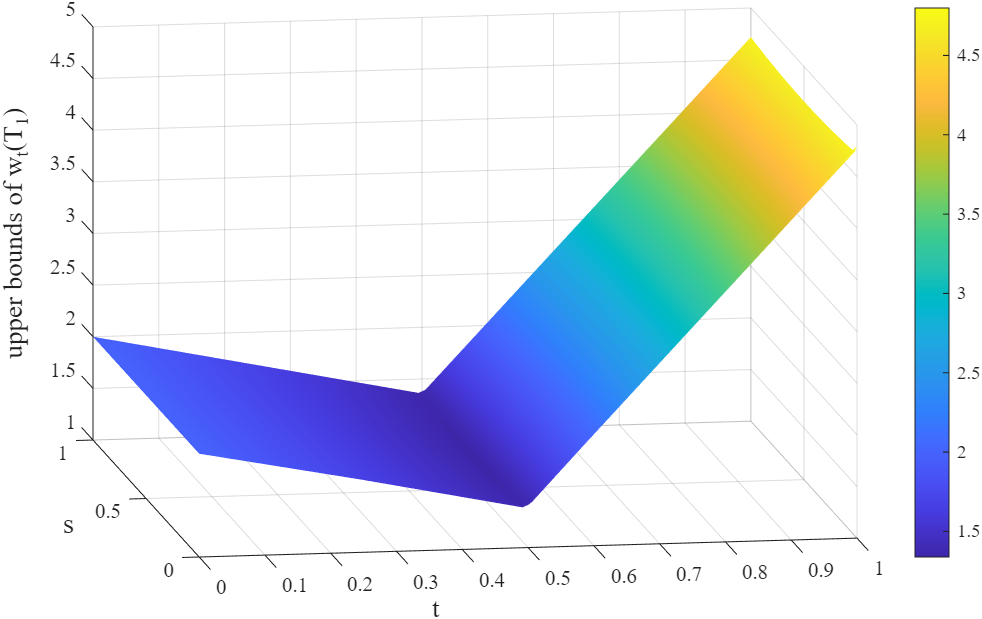}
\includegraphics[width=4.5cm, height=4.5cm]{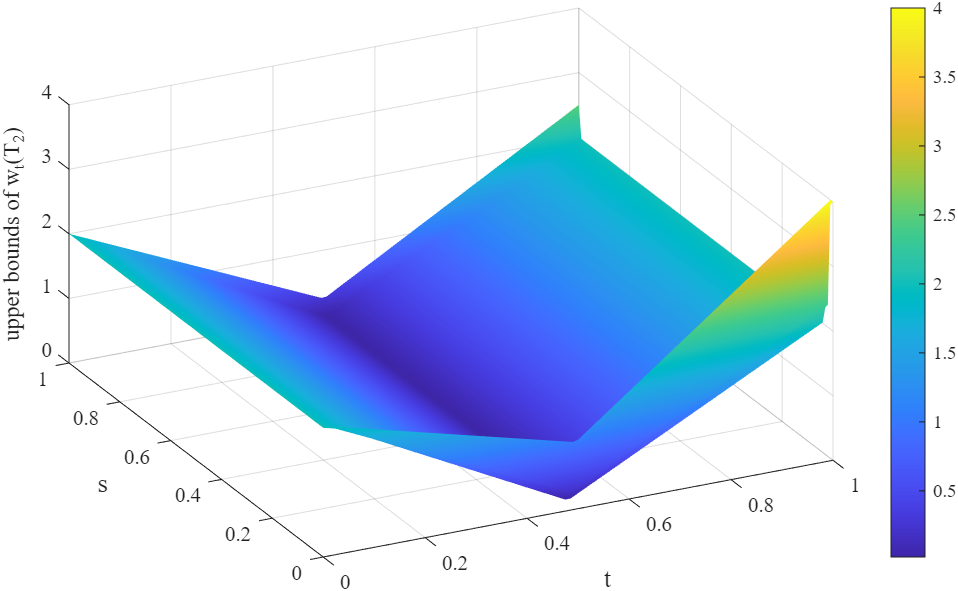}
\includegraphics[width=4.5cm, height=4.5cm]{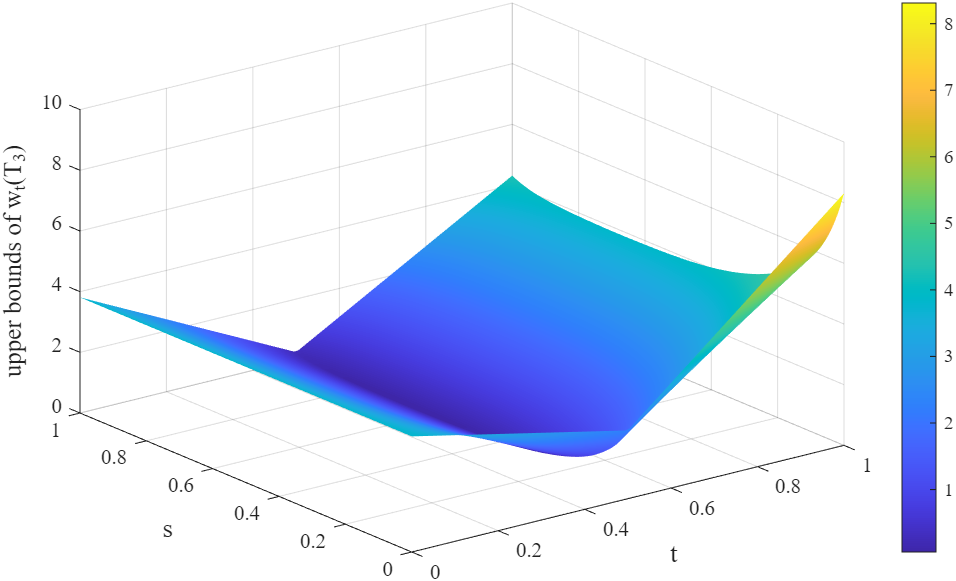}
\caption{The upper bounds of $w_t(T)$ for different types of matrices with $\mathsf{s}$-Aluthge transformation }
    \label{fig:01}
\end{figure}

\begin{figure}[h]
    \centering
    \includegraphics[width=4.5cm, height=4.5cm]{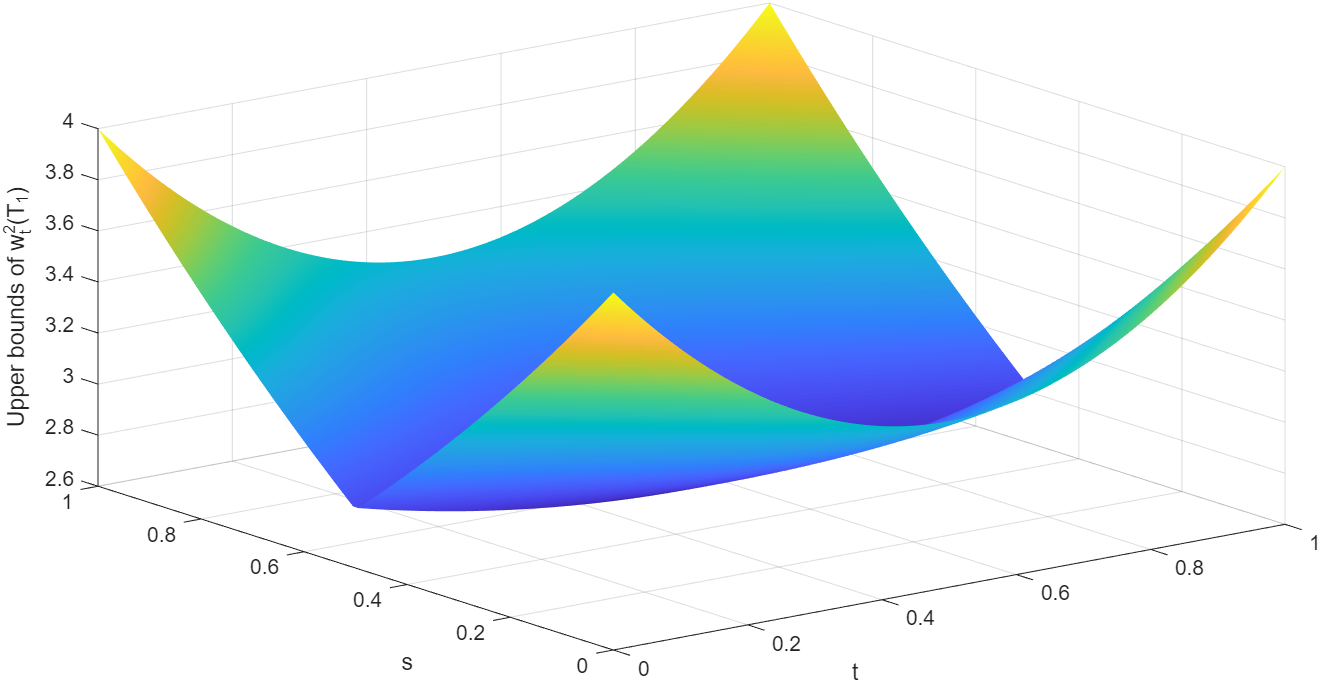}
    \includegraphics[width=4.5cm, height=4.5cm]{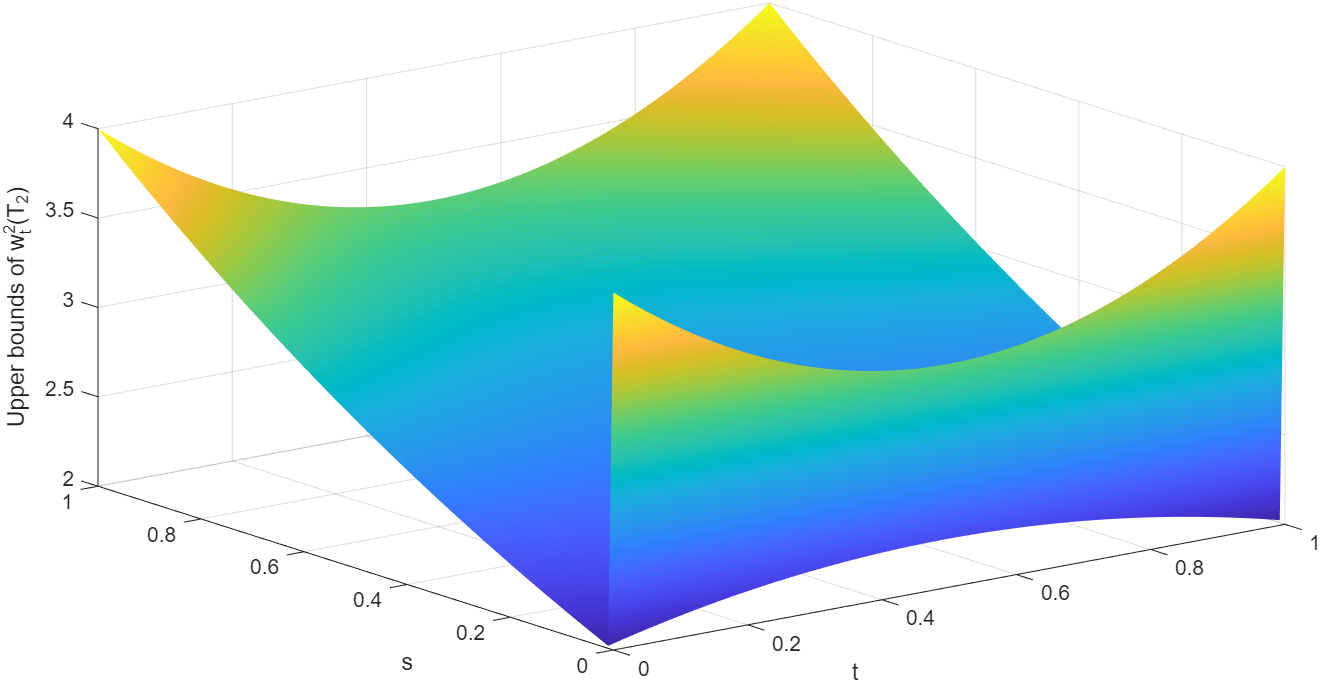}
    \includegraphics[width=4.5cm, height=4.5cm]{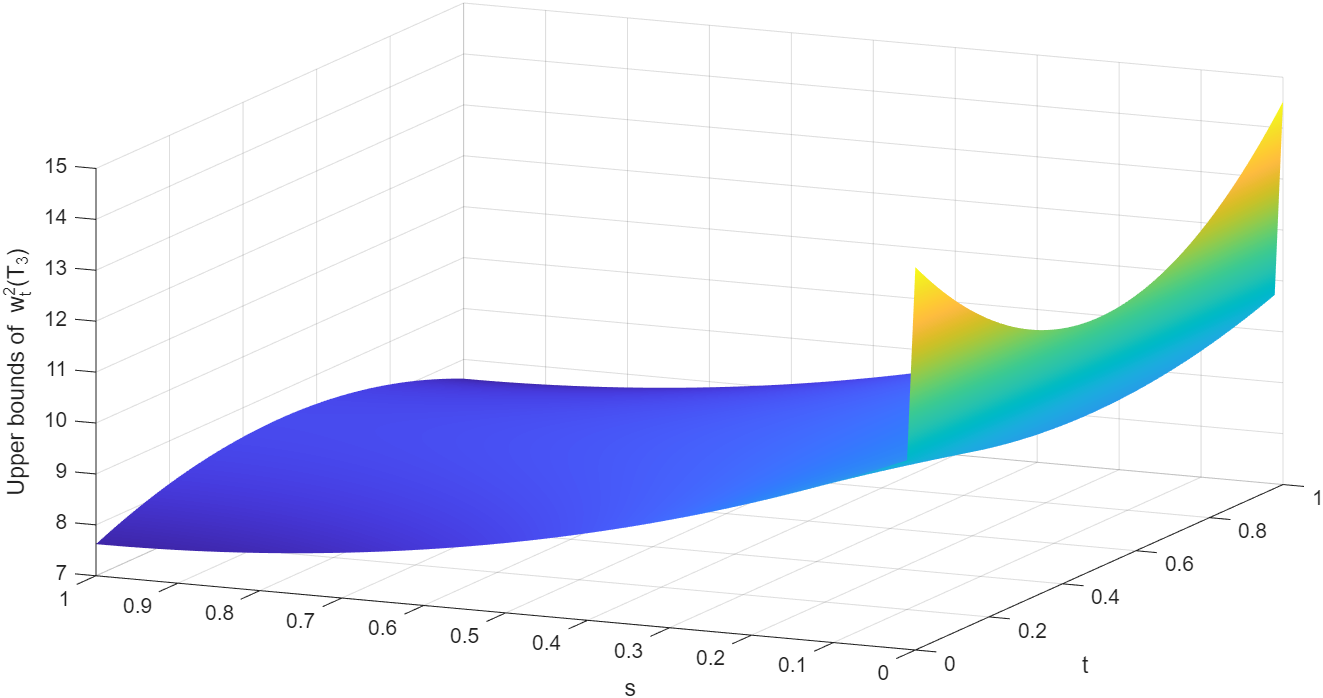}
    \caption{The upper bounds of $w_t^2(T)$ for different types of matrices with $\mathsf{s}$-Aluthge transformation }
    \label{fig:02}
\end{figure}
\end{ex}

\newpage

\begin{theorem}\label{te4}
For $1 \leq p\leq \infty$, $\mathsf{s},t \in [0,1]$ and $T \in B_p(H)$, we have
 \begin{align*}
w_{t,4p}^4(T) &\leq \{t^2+(1-t)^2\}\|T\|^2_{2p}\|\widetilde{T}_\mathsf{s}\|^2_{2p} \\
&\qquad
+
t^2\left\|
T^2K+KT^2
\right\|_{p}+(1-t)^2\left\|(T^*)^2K+K(T^*)^2\right\|_{p}
+\|K\|^2_{2p},
\end{align*}
where $K=T^*T+TT^*$.
\end{theorem}
\begin{proof}
Consider,
\begin{align*}
&\left\|
\mathfrak{R}_t(e^{i\theta}T)
\right\|_{4p}^4\\
&=\left\| \left|te^{i\theta}T + (1-t)e^{-i\theta}T^*\right|^2\right\|^2_{2p}\\
&= 
\big\| t^2e^{2i\theta}T^2
+ (1-t)^2e^{-2i\theta}(T^*)^2
+ t(1-t)(TT^* + T^*T)\big\|^2_{2p}\\
&= 
\big\|t^2e^{2i\theta}T^2
+ (1-t)^2e^{-2i\theta}(T^*)^2+K\big\|^2_{2p}
\hspace{1cm}(~\text{taking}~ K=T^*T+TT^*)\\
&=
\big\|\left|
t^2e^{2i\theta}T^2 + (1-t)^2e^{-2i\theta}(T^*)^2 + K
\right|^2\big\|_{p}\\
&=
\Bigl\|
\left(
t^2e^{2i\theta}T^2 + (1-t)^2e^{-2i\theta}(T^*)^2
\right)^2+\left(
t^2e^{2i\theta}T^2 + (1-t)^2e^{-2i\theta}(T^*)^2
\right)K
\\
\nonumber &\qquad\quad
+K\left(
t^2e^{2i\theta}T^2 + (1-t)^2e^{-2i\theta}(T^*)^2
\right)
+K^2
\Bigr\|_{p}\\
 &=
\Bigl\|
\left(
t^2e^{2i\theta}T^2 + (1-t)^2e^{-2i\theta}(T^*)^2
\right)^2+t^2e^{2i\theta}T^2 K
+(1-t)^2e^{-2i\theta}(T^*)^2 K
\\
&\qquad\quad
+Kt^2e^{2i\theta}T^2
+K(1-t)^2e^{-2i\theta}(T^*)^2
+K^2
\Bigr\|_{p}
\\
&=
\Bigl\|
\left(
t^2e^{2i\theta}T^2 + (1-t)^2e^{-2i\theta}(T^*)^2
\right)^2
\\
&\qquad\quad
+t^2e^{2i\theta}(T^2K+KT^2)
+(1-t)^2e^{-2i\theta}\bigl((T^*)^2K+K(T^*)^2\bigr)
+K^2
\Bigr\|_{p}\\
&\le
\Bigl\|
\big|t^2e^{2i\theta}T^2+(1-t)^2e^{-2i\theta}(T^*)^2\big|^2
\Bigr\|_{p}
\\
&
+
\left\|
t^2e^{2i\theta}(T^2K+KT^2)
\right\|_{p}+\left\|(1-t)^2e^{-2i\theta}\bigl((T^*)^2K+K(T^*)^2\bigr)\right\|_{p}
+\||K|^2\|_{p}
\\
&=
\Bigl\|
t^2e^{2i\theta}U|T|U|T|
+
(1-t)^2e^{-2i\theta}|T|U^*|T|U^*
\Bigr\|^2_{2p}
\\
&\qquad
+
t^2\left\|
(T^2K+KT^2)
\right\|_{p}+(1-t)^2\left\|(T^*)^2K+K(T^*)^2\right\|_{p}
+\|K\|^2_{2p}
\\
&= \left\|t^2e^{2i\theta}U|T|^{1-\mathsf{s}}|T|^\mathsf{s} U|T|^{1-\mathsf{s}}|T|^\mathsf{s} + (1-t)^2e^{-2i\theta}|T|^\mathsf{s}|T|^{1-\mathsf{s}}U^*|T|^\mathsf{s} |T|^{1-\mathsf{s}}U^*\right\|^2_{2p} \\
&\qquad
+
t^2\left\|
(T^2K+KT^2)
\right\|_{p}+(1-t)^2\left\|(T^*)^2K+K(T^*)^2\right\|_{p}
+\|K\|^2_{2p}
\\
&=\left\| t^2e^{2i\theta}U|T|^{1-\mathsf{s}}\widetilde{T}_\mathsf{s}|T|^\mathsf{s} 
+ (1-t)^2e^{-2i\theta}|T|^\mathsf{s}\widetilde{T}_\mathsf{s}^*|T|^{1-\mathsf{s}}U^* \right\|^2_{2p}\\
&\qquad
+
t^2\left\|
(T^2K+KT^2)
\right\|_{p}+(1-t)^2\left\|(T^*)^2K+K(T^*)^2\right\|_{p}
+\|K\|^2_{2p}\\
&\leq \Bigl(t^2\left\| e^{2i\theta}U|T|^{1-\mathsf{s}}\widetilde{T}_\mathsf{s}|T|^\mathsf{s} \right\|_{2p}
+ (1-t)^2\left\| e^{-2i\theta}|T|^\mathsf{s}\widetilde{T}_\mathsf{s}^*|T|^{1-\mathsf{s}}U^* \right\|_{2p}\Bigr)^2\\
&\qquad
+
t^2\left\|
T^2K+KT^2
\right\|_{p}+(1-t)^2\left\|(T^*)^2K+K(T^*)^2\right\|_p
+\|K\|^2_{2p}
\\
&\leq \{t^2+(1-t)^2\}^2\|T\|^2_{2p}\|\widetilde{T}_\mathsf{s}\|^2_{2p} \\
&\qquad
+
t^2\left\|
T^2K+KT^2
\right\|_{p}+(1-t)^2\left\|(T^*)^2K+K(T^*)^2\right\|_{p}
+\|K\|^2_{2p}.
\end{align*}
Taking supremum over \(\theta \in \mathbb{R}\) in the above inequality, we get
\begin{align*}
w_{t,4p}^4(T) &\leq \{t^2+(1-t)^2\}^2\|T\|^2_{2p}\|\widetilde{T}_\mathsf{s}\|^2_{2p} \\
&\qquad
+
t^2\left\|
(T^2K+KT^2
\right\|_{p}+(1-t)^2\left\|(T^*)^2K+K(T^*)^2\right\|_{p}
+\|K\|^2_{2p}.
\end{align*}
\end{proof}
\noindent Taking $p=\infty$ in the Theorem \ref{te4}, we get
\begin{corollary}
 Let $T\in B(H)$. Then   
 \begin{align*}
w_t^4(T) &\leq \{t^2+(1-t)^2\}\|T\|^2\|\widetilde{T}_\mathsf{s}\|^2 \\
&\qquad
+
t^2\left\|
(T^2K+KT^2)
\right\|+(1-t)^2\left\|(T^*)^2K+K(T^*)^2\right\|
+\|K\|^2,
\end{align*}
where $K=T^*T+TT^*$.
\end{corollary}

 For $p=\infty$, the Schatten \(p\)-norm is the usual operator norm.
In the following result, we obtain an upper bound for the numerical
radius involving general powers.
\begin{theorem}
For $T\in B(H)$, $\mathsf{s}\in [0,1]$ and $m\ge 1$, we have
\[
w_t^{2m}(T)
\leq 2^{2(m-1)}\Big(\{t^{2m}+(1-t)^{2m}\}\|T\|^m\|\widetilde{T}_\mathsf{s}\|^m+t^m(1-t)^m\left\|
{(T^*T)^m+(TT^*)^m}
\right\|\Big).
\]    
\end{theorem}

\begin{proof}
From \eqref{1w}, we get
\begin{align*}
(\Re_t(e^{i\theta}T) )^2
&= t^2e^{2i\theta}T^2 + (1-t)^2e^{-2i\theta}T^{*2} +t(1-t)(T^*T + TT^*).
\end{align*}
Therefore,
\begin{align*}
\|(\Re_t(e^{i\theta}T) )^2\|
&\le\left\|t^2e^{2i\theta}T^2 + (1-t)^2e^{-2i\theta}T^{*2}\right\|
+t(1-t)\|T^*T+TT^*\|.
\end{align*}
Now,
\begin{align*}
&\|(\Re_t(e^{i\theta}T) )^2\|^m=\|\Re_t(e^{i\theta}T) \|^{2m}\\
&\le
\Big(
\left\|t^2e^{2i\theta}T^2 + (1-t)^2e^{-2i\theta}T^{*2}\right\|
+t(1-t)\|T^*T+TT^*\|
\Big)^m \\
&\le
2^{p-1}\Big(\left\|t^2e^{2i\theta}T^2 + (1-t)^2e^{-2i\theta}T^{*2}\right\|^m
+t^m(1-t)^m\left\|{T^*T+TT^*}\right\|^m\Big) \\
&=2^{m-1}\Bigg(\Bigl\|
t^2e^{2i\theta}U|T|U|T|
+
(1-t)^2e^{-2i\theta}|T|U^*|T|U^*
\Bigr\|^m\\
&\hspace{5cm}+2^m t^m(1-t)^m\left\|\frac{T^*T+TT^*}{2}\right\|^m\Bigg)\\
&\leq2^{m-1}\Bigg(\Big\|t^2e^{2i\theta}U|T|^{1-\mathsf{s}}|T|^\mathsf{s} U|T|^{1-\mathsf{s}}|T|^\mathsf{s} + (1-t)^2e^{-2i\theta}|T|^\mathsf{s}|T|^{1-\mathsf{s}}U^*|T|^\mathsf{s} |T|^{1-\mathsf{s}}U^*\Big\|^m\\
&\hspace{3cm}+2^m t^m(1-t)^m\left\|
\left(
\frac{(T^*T)^m+(TT^*)^m}{2}
\right)^{\frac{1}{m}}
\right\|^m\Bigg) \\
&\leq2^{m-1}\Bigg(\left\| t^2e^{2i\theta}U|T|^{1-\mathsf{s}}\widetilde{T}_\mathsf{s}|T|^\mathsf{s} 
+ (1-t)^2e^{-2i\theta}|T|^\mathsf{s}\widetilde{T}_\mathsf{s}^*|T|^{1-\mathsf{s}}U^* \right\|^m\\
&\hspace{6cm}+2^m t^m(1-t)^m\left\|
\frac{(T^*T)^m+(TT^*)^m}{2}
\right\|\Bigg)\\
&\leq2^{m-1}\Bigg(2^{m-1}\Big(\left\| t^2e^{2i\theta}U|T|^{1-\mathsf{s}}\widetilde{T}_\mathsf{s}|T|^\mathsf{s}\right\| 
+ \left\|(1-t)^2e^{-2i\theta}|T|^\mathsf{s}\widetilde{T}_\mathsf{s}^*|T|^{1-\mathsf{s}}U^* \right\|^m\Big)\\
&\hspace{6cm}+2^{m-1} t^m(1-t)^m\left\|
{(T^*T)^m+(TT^*)^m}
\right\|\Bigg)\\
&=2^{2(m-1)}\Bigg(t^{2m}\left\| U|T|^{1-\mathsf{s}}\widetilde{T}_\mathsf{s}|T|^\mathsf{s}\right\| 
+ (1-t)^{2m}\left\||T|^\mathsf{s}\widetilde{T}_\mathsf{s}^*|T|^{1-\mathsf{s}}U^* \right\|^m\\
&\hspace{6cm}+t^m(1-t)^m\left\|
{(T^*T)^m+(TT^*)^m}
\right\|\Bigg)\\
&\leq 2^{2(m-1)}\Big(\{t^{2m}+(1-t)^{2m}\}\|T\|^m\|\widetilde{T}_\mathsf{s}\|^m+t^m(1-t)^m\left\|
{(T^*T)^m+(TT^*)^m}
\right\|\Big).
\end{align*}
\noindent Taking the supremum over $\theta\in\mathbb{R}$ in the above inequality, we get

\[
w_t^{2m}(T)
\leq 2^{2(m-1)}\Big(\{t^{2m}+(1-t)^{2m}\}\|T\|^m\|\widetilde{T}_\mathsf{s}\|^m+t^m(1-t)^m\left\|
{(T^*T)^m+(TT^*)^m}
\right\|\Big).
\]
\end{proof}

\begin{theorem}\label{k1}
 For $1 \leq p\leq \infty$, $t \in [0,1]$ and $T,S \in B_p(H)$,  if $f,g$ are non-negative continuous functions on $[0,\infty)$ such that $f(x)g(x)=x$, then 
we have
\begin{align*}
w_{t,p}(T+S)\leq &2^{\frac1p}t
\left[\left\|
\Re(e^{i\theta}\widetilde{T}_{f,g})
\right\|_p
+\left\|
\Re(e^{i\theta}\widetilde{S}_{f,g})
\right\|_p\right]+|1-2t|\big(\|T\|_p+\|S\|_p)\\
&\hspace{1cm}+2^{\frac1p-1}t
\left[
\left\|
f^2(|T|)+g^2(|T|)
\right\|_p+
\left\|
f^2(|S|)+g^2(|S|)
\right\|_p\right].
\end{align*}
\end{theorem}
\begin{proof}
Let $T=U|T|$ and $S=V|S|$ be the polar decompositions of $T$ and $S$, respectively.
Consider,
Consider,
\begin{align*}
&\|\Re_t(e^{i\theta}(T+S)\|_p\\
&=\|\Re_t(e^{i\theta}(T)+\Re_t(e^{i\theta}(S)\|_p\\
&=\|\Re_t(e^{i\theta}(T)\|_p+\|\Re_t(e^{i\theta}(S)\|_p\\
&\leq 2^{\frac1p}t
\left[\left\|
\Re(e^{i\theta}\widetilde{T}_{f,g})
\right\|_p
+
\frac{1}{2}
\left\|
f^2(|T|)+g^2(|T|)
\right\|_p\right]+|1-2t|\|T\|_p\\
&\hspace{1cm}+2^{\frac1p}t
\left[\left\|
\Re(e^{i\theta}\widetilde{S}_{f,g})
\right\|_p
+
\frac{1}{2}
\left\|
f^2(|S|)+g^2(|S|)
\right\|_p\right]+|1-2t|\|S\|_p\\
&\hspace{7cm} (\text{using}~ \eqref{k2})\\
&\leq 2^{\frac1p}t
\left[\left\|
\Re(e^{i\theta}\widetilde{T}_{f,g})
\right\|_p
+\left\|
\Re(e^{i\theta}\widetilde{S}_{f,g})
\right\|_p\right]\\
&\hspace{1cm}+2^{\frac1p-1}t
\left[
\left\|
f^2(|T|)+g^2(|T|)
\right\|_p+
\left\|
f^2(|S|)+g^2(|S|)
\right\|_p\right]+|1-2t|\big(\|T\|_p+\|S\|_p).
\end{align*}
Taking supremum over \(\theta \in \mathbb{R}\) in the above inequality, we get
\begin{align*}
w_{t,p}(T+S)\leq &2^{\frac1p}t
\left[\left\|
\Re(e^{i\theta}\widetilde{T}_{f,g})
\right\|_p
+\left\|
\Re(e^{i\theta}\widetilde{S}_{f,g})
\right\|_p\right]+|1-2t|\big(\|T\|_p+\|S\|_p)\\
&\hspace{1cm}+2^{\frac1p-1}t
\left[
\left\|
f^2(|T|)+g^2(|T|)
\right\|_p+
\left\|
f^2(|S|)+g^2(|S|)
\right\|_p\right].
\end{align*}
\end{proof}
\noindent Taking $g(x)=x^\mathsf{s}$ and $f(x)=x^{1-\mathsf{s}}$ in \eqref{0c1}, we get the following result.
\begin{theorem}
 For $1 \leq p\leq \infty$, $\mathsf{s},t \in [0,1]$ and $T \in B_p(H)$, we have
\begin{align*}
w_{t,p}(T+S)\leq &2^{\frac1p}t
\left[
w_p(\widetilde{T}_{1-\mathsf{s}})+
w_p(\widetilde{S}_{1-\mathsf{s}})\right]+|1-2t|[\|T\|_p+\|S\|_p]\\
&+2^{\frac1p-1}t\left[
\left\|
|T|^{2\mathsf{s}}+|T|^{2(1-\mathsf{s})}
\right\|_p+\left\|
|S|^{2\mathsf{s}}+|S|^{2(1-\mathsf{s})}
\right\|_p\right].
\end{align*}
\end{theorem}
Taking $t=\frac{1}{2}$ in the Theorem \ref{k1}, we obtain the following result.
\begin{corollary}
For $1 \leq p\leq \infty$, $\mathsf{s} \in [0,1]$ and $T, S \in B_p(H)$, we have
\begin{align*}
w_{p}(T+S)\leq &2^{\frac1p-1}
\left[
w_p(\widetilde{T}_{1-\mathsf{s}})+
w_p(\widetilde{S}_{1-\mathsf{s}})\right]\\
&+2^{\frac1p-2}\left[
\left\|
|T|^{2\mathsf{s}}+|T|^{2(1-\mathsf{s})}
\right\|_p+\left\|
|S|^{2\mathsf{s}}+|S|^{2(1-\mathsf{s})}
\right\|_p\right].
\end{align*}  
Moreover, for $\mathsf{s=\frac{1}{2}}$, 
\begin{align*}
w_{p}(T+S)\leq &2^{\frac1p-1}
\left[
w_p(\widetilde{T})+
w_p(\widetilde{S})\right]+2^{\frac1p-1}[\|T\|_p+\|S\|_p].
\end{align*}
\end{corollary}
For $p=\infty$ in the Theorem \ref{k1}, we get
\begin{corollary}
For $T, S \in B(H)$ and $\mathsf{s},t \in [0,1]$, we have
\begin{align*}
w_t(T+S)\leq & t
\left[
w(\widetilde{T}_{1-\mathsf{s}})+
w(\widetilde{S}_{1-\mathsf{s}})\right]+|1-2t|[\|T\|+\|S\|]\\
&+\frac{1}{2}t\left[
\left\|
|T|^{2\mathsf{s}}+|T|^{2(1-\mathsf{s})}
\right\|+\left\|
|S|^{2\mathsf{s}}+|S|^{2(1-\mathsf{s})}
\right\|\right].
\end{align*}  
Moreover, for $\mathsf{s=\frac{1}{2}}$, 
\begin{align*}
w_t(T+S)\leq t
\left[
w(\widetilde{T})+
w(\widetilde{S})\right]+\{|1-2t|+t\}[\|T\|+\|S\|].\\
\end{align*}  
\end{corollary}

\begin{theorem} \label{b5} 
 For $1 \leq p\leq \infty$, $t \in [0,1]$ and $T,S \in B_p(H)$,  If $f,g$ are non-negative continuous functions on $[0,\infty)$ such that $f(x)g(x)=x$, 
we have
\begin{align}\label{e5}
w_{t,p}\!\left(
\begin{bmatrix}
0 & T\\
S & 0
\end{bmatrix}
\right)
\nonumber&\le
  2^\frac{1}{p}t\Bigg[
2^{\frac{1}{p}-1}\left(\bigl\|f(|S|)g(|T^*|)\bigr\|_p+\bigl\|f(|T|)g(|S^*|)\bigr\|_p
\right)\\
&\nonumber\hspace{2cm}+\frac{1}{2}\Bigl(\|f^2(|S|)+g^2(|S|)\|_p^p+\|f^2(|T|)+g^2(|T|)\|_p^p\Bigr)^{\frac1p}\Bigg]\\
&\hspace{2cm}+|1-2t|
\bigl(\|T\|_p^p+\|S\|_p^p\bigr)^{\frac1p}.
\end{align}
\end{theorem}
\begin{proof}
Let $T=U|T|$ and $S=V|S|$ be the polar decompositions of $T$ and $S$, respectively.
For $\mathbb{E}=\begin{bmatrix}
0 & T\\
S & 0
\end{bmatrix}$, we have
\begin{align*}\Tilde{\mathbb{E}}_{f,g}&=f(|\mathbb{E}|)\begin{bmatrix}
0 & U\\
V & 0
\end{bmatrix}g(|\mathbb{E}|)\\
&=\begin{bmatrix}
f(|S|) & 0\\
0 & f(|T|)
\end{bmatrix}
\begin{bmatrix}
0 & U\\
V & 0
\end{bmatrix}
\begin{bmatrix}
g(|S|) & 0\\
0 & g(|T|)
\end{bmatrix}
\end{align*}
and
\begin{align*}
\left\|
f^2(|\mathbb{E}|)+g^2(|\mathbb{E}|)
\right\|_p=\left\|\begin{bmatrix}
f^2(|S|)+g^2(|S|) & 0\\
0 & f^2(|T|)+g^2(|T|)
\end{bmatrix}\right\|_p.
\end{align*}
From \eqref{0c1}, we obtain
\begin{align*}
&w_{t,p}\left(
\begin{bmatrix}
0 & T\\
S & 0
\end{bmatrix}
\right)\\
&\leq2^{\frac1p}t\left[
w_p(\Tilde{\mathbb{E}}_{f,g})
+
\frac{1}{2}
\left\|
f^2(|\mathbb{E}|)+g^2(|\mathbb{E}|)
\right\|_p\right]+|1-2t|\|\mathbb{E}\|_p\\
&\le 2^\frac{1}{p}t\Bigg[
w_p\!\left(
\begin{bmatrix}
f(|S|) & 0\\
0 & f(|T|)
\end{bmatrix}
\begin{bmatrix}
0 & U\\
V & 0
\end{bmatrix}
\begin{bmatrix}
g(|S|) & 0\\
0 & g(|T|)
\end{bmatrix}
\right)\\
&\hspace{2cm}+\frac{1}{2}\left\|\begin{bmatrix}
f^2(|S|)+g^2(|S|) & 0\\
0 & f^2(|T|)+g^2(|T|)
\end{bmatrix}\right\|_p\Bigg]+|1-2t|
\left\|
\begin{bmatrix}
0 & T\\
S & 0
\end{bmatrix}
\right\|_p\\
&=
2^\frac{1}{p}t\Bigg[
w_p\!\left(
\begin{bmatrix}
0 & f(|S|)Ug(|T|)\\
f(|T|)Vg(|S|) & 0
\end{bmatrix}
\right)\\
&\hspace{2cm}+\frac{1}{2}\Bigl(\|f^2(|S|)+g^2(|S|)\|_p^p+\|f^2(|T|)+g^2(|T|)\|_p^p\Bigr)^{\frac1p}\Bigg]+ |1-2t|
\bigl(\|T\|_p^p+\|S\|_p^p\bigr)^{\frac1p}
\\
&\le
 2^\frac{1}{p}t\Bigg[
w_p\!\left(
\begin{bmatrix}
0 & f(|S|)Ug(|T|)\\
0 & 0
\end{bmatrix}
\right)
+
w_p\!\left(
\begin{bmatrix}
0 & 0\\
f(|T|)Vg(|S|) & 0
\end{bmatrix}
\right)\\
&\hspace{2cm}+\frac{1}{2}\Bigl(\|f^2(|S|)+g^2(|S|)\|_p^p+\|f^2(|T|)+g^2(|T|)\|_p^p\Bigr)^{\frac1p}\Bigg]+|1-2t|
\bigl(\|T\|_p^p+\|S\|_p^p\bigr)^{\frac1p}
\\
&=
 2^\frac{1}{p}t\Big[2^{\frac{1}{p}-1}
\bigl\|f(|S|)Ug(|T|)\bigr\|_p
+
2^{\frac{1}{p}-1}\bigl\|f(|T|)Vg(|S|)\bigr\|_p\\
&\hspace{1cm}+\frac{1}{2}\Bigl(\|f^2(|S|)+g^2(|S|)\|_p^p+\|f^2(|T|)+g^2(|T|)\|_p^p\Bigr)^{\frac1p}\Big]+|1-2t|
\bigl(\|T\|_p^p+\|S\|_p^p\bigr)^{\frac1p}
\\
&\hspace{7cm} (\text{by the inequality} \eqref{2})\\
&=
 2^\frac{1}{p}t\Bigg[
2^{\frac{1}{p}-1}\left(\bigl\|f(|S|)UU^*g(|T^*|)U\bigr\|_p+
\bigl\|f(|T|)VV^*g(|S^*|)V\bigr\|_p
\right)\\
&\hspace{1cm}+\frac{1}{2}\Bigl(\|f^2(|S|)+g^2(|S|)\|_p^p+\|f^2(|T|)+g^2(|T|)\|_p^p\Bigr)^{\frac1p}\Bigg]+|1-2t|
\bigl(\|T\|_p^p+\|S\|_p^p\bigr)^{\frac1p}\\
&\hspace{3cm}(\text{since}~ g(|T|)=U^*g(|T^*|)U, ~\text{and}~ g(|S|)=V^*g(|S^*|)V 
)\\
&=
 2^\frac{1}{p}t\Bigg[
2^{\frac{1}{p}-1}\left(\bigl\|f(|S|)g(|T^*|)U\bigr\|_p+
\bigl\|f(|T|)g(|S^*|)V\bigr\|_p
\right)\\
&\hspace{1cm}+\frac{1}{2}\Bigl(\|f^2(|S|)+g^2(|S|)\|_p^p+\|f^2(|T|)+g^2(|T|)\|_p^p\Bigr)^{\frac1p}\Bigg]+|1-2t|
\bigl(\|T\|_p^p+\|S\|_p^p\bigr)^{\frac1p}\\
&\leq
 2^\frac{1}{p}t\Bigg[
2^{\frac{1}{p}-1}\left(\bigl\|f(|S|)g(|T^*|)\bigr\|_p+
\bigl\|f(|T|)g(|S^*|)\bigr\|_p
\right)\\
&\hspace{1cm}+\frac{1}{2}\Bigl(\|f^2(|S|)+g^2(|S|)\|_p^p+\|f^2(|T|)+g^2(|T|)\|_p^p\Bigr)^{\frac1p}\Bigg]+|1-2t|
\bigl(\|T\|_p^p+\|S\|_p^p\bigr)^{\frac1p}\\
&\hspace{7cm} (\text{by the inequality} \eqref{2}).
\end{align*}
Hence the proved.
\end{proof}
Taking $t=\frac{1}{2}$ in \eqref{e5}, we get the following result of S. Sahoo \cite{sahoo2022further}.
\begin{corollary}
For $1 \leq p\leq \infty$, $\mathsf{s},t \in [0,1]$ and $T, S \in B_p(H)$, we have    
\begin{align}
w_{t,p}\!\left(
\begin{bmatrix}
0 & T\\
S & 0
\end{bmatrix}
\right)
\nonumber&\le
  2^{\frac{1}{p}-1}\Bigg[
2^{\frac{1}{p}-1}\left(\bigl\|f(|S|)g(|T^*|)\bigr\|_p+\bigl\|f(|T|)g(|S^*|)\bigr\|_p
\right)\\
&\nonumber\hspace{2cm}+\frac{1}{2}\Bigl(\|f^2(|S|)+g^2(|S|)\|_p^p+\|f^2(|T|)+g^2(|T|)\|_p^p\Bigr)^{\frac1p}\Bigg].
\end{align}
\end{corollary}
\noindent Taking $g(x)=x^\mathsf{s}$ and $f(x)=x^{1-\mathsf{s}}$ in the Theorem \ref{b5}, we get the following result.
\begin{theorem}  
For $1 \leq p\leq \infty$, $\mathsf{s},t \in [0,1]$ and $\mathsf{T}, \mathsf{S} \in \mathsf{B}_p(\mathsf{H})$, we have
\begin{align}\label{p05}
w_{t,p}\!\left(
\begin{bmatrix}
0 & T\\
S & 0
\end{bmatrix}
\right)
\nonumber&\le
 2^\frac{1}{p}t\left[
\left(
\bigl\||\mathsf{T}|^{\frac12}|\mathsf{S}^*|^{\frac12}\bigr\|_p
+
\bigl\||\mathsf{S}|^{\frac12}|\mathsf{T}^*|^{\frac12}\bigr\|_p
\right)+\bigl(\|\mathsf{T}\|_p^p+\|\mathsf{S}\|_p^p\bigr)^{\frac1p}\right]\\
&\hspace{3cm}+|1-2t|
\bigl(\|\mathsf{T}\|_p^p+\|\mathsf{S}\|_p^p\bigr)^{\frac1p}.  
\end{align}
\end{theorem}

\noindent Taking $t=\frac{1}{2}$ in \eqref{p05}, we get the following result of Frakis et al. \cite{frakis2025p}.
\begin{corollary}
For $1 \leq p\leq \infty$, $\mathsf{s} \in [0,1]$ and $T \in B_p(H)$, we get
\begin{align}
w_{p}\!\left(
\begin{bmatrix}
0 & T\\
S & 0
\end{bmatrix}
\right)
\le
2^{\frac{2}{p}-2}
\left(
\bigl\||T|^{\frac12}|S^*|^{\frac12}\bigr\|_p
+
\bigl\||S|^{\frac12}|T^*|^{\frac12}\bigr\|_p
\right)
+
2^{\frac1p-1}
\left(
\|T\|_p^p+\|S\|_p^p
\right)^{\frac1p}.    
\end{align}
\end{corollary}

\noindent\textbf{Declaration:}
 
\noindent\textbf{Competing Interests:} The authors declare that they have no competing interests. 
 
\noindent \textbf{Authors' contributions:}
All authors have contributed equally to this work.  

\noindent\textbf{Funding}: The authors declare that no funds, grants, or other support was received during the preparation of this manuscript.

\end{document}